\documentclass{amsart}

\newtheorem{theorem}{Theorem}[section]
\newtheorem{lemma}[theorem]{Lemma}

\theoremstyle{definition}
\newtheorem{definition}[theorem]{Definition}
\newtheorem{example}[theorem]{Example}
\newtheorem{remark}[theorem]{Remark}

\newcommand{\bF}{\bar{F}}
\newcommand{\bL}{\bar{L}}
\newcommand{\bS}{\bar{S}}
\newcommand{\bV}{\bar{V}}
\newcommand{\bW}{\bar{W}}
\newcommand{\fF}{{\mathfrak F}}

\newcommand{\p}{{\mbox{$[p]$}}}
\newcommand{\pd}{{\mbox{$[p]'$}}}
\newcommand{\id}{\triangleleft}
\DeclareMathOperator{\ideq}{\unlhd}
\newcommand{\sn}{\mbox{$\triangleleft\mspace{-1.8mu}\triangleleft \medspace$}}
\newcommand{\scp}{{\mbox{$\scriptstyle [p]$}}}

\DeclareMathOperator{\cl}{Cl}
\DeclareMathOperator{\ind}{Ind}
\DeclareMathOperator{\Hom}{Hom}
\DeclareMathOperator{\End}{End}
\DeclareMathOperator{\res}{Res}
\DeclareMathOperator{\ammod}{AmMod}

\numberwithin{equation}{section}

\title[Character clusters]{Character clusters for Lie algebra modules over a field of non-zero characteristic}
\author{Donald W. Barnes}
\address{1 Little Wonga Rd.\\Cremorne NSW 2090\\Australia\\}
\email{donwb@iprimus.com.au}

\begin{document}

\begin{abstract} For a Lie algebra $L$ over an algebraically closed field $F$ of non-zero characteristic, every finite dimensional $L$-module can be decomposed into a direct sum of submodules such that all composition factors of a summand have the same character.  Using the concept of a character cluster, this result is generalised to fields which are not algebraically closed.  Also, it is shown that if the soluble Lie algebra $L$ is in the saturated formation $\fF$ and if $V, W$ are irreducible $L$-modules with the same cluster and the $p$-operation vanishes on the centre of the $p$-envelope used, then $V,W$ are either both $\fF$-central or both $\fF$-eccentric.  Clusters are used to generalise the construction of induced modules.

\end{abstract}

\maketitle
\section{Introduction}
Throughout this note, $L$ is a finite dimensional Lie algebra over the field $F$ of characteristic $p \ne 0$.  Let $V$ be a finite dimensional $L$-module.  To define a character for $V$, we must embed $L$ in a $p$-envelope $(L^p, \p)$.  The action $\rho$ of $L$ on $V$ can be extended to $L^p$.  (See Strade and Farnsteiner \cite[Theorem 5.1.1]{SF}.)

\begin{definition}  A character for $V$ is a linear map $c: L^p \to F$ such that for all $x \in L^p$, we have 
$$\rho(x)^p - \rho(x^{\scp}) = c(x)^p1.$$
\end{definition}
 Not every module has a character, but if $F$ is algebraically closed and $V$ is irreducible, then $V$ has a character.  (See Strade and Farnsteiner \cite[Theorem 5.2.5]{SF}.)  The following is Strade and Farnsteiner, \cite[Theorem 5.2.6]{SF}.

\begin{theorem} \label{SFdecomp}Suppose that $F$ is algebraically closed and let $(L,\p)$ be a restricted Lie algebra over $F$.  Let $V$ be a finite dimensional $L$-module.  Then there exist $c_i:L \to F$ and submodules $V_i$ such that $V = \oplus_i V_i$ and every composition factor of $V_i$ has character $c_i$.
\end{theorem}

This decomposition in terms of characters is functorial and is clearly useful.  In this note, the concept of a character cluster is used to obtain a similar result which does not require the field to be algebraically closed.
As a further application, it is shown that, if the soluble Lie algebra $L$ is in the saturated formation $\fF$ and $V, W$ are irreducible $L$-modules with the same cluster and the $p$-operation vanishes on the centre of the $p$-envelope used, then either both $V,W$ are $\fF$-central or both are $\fF$-eccentric.  Over a perfect field, clusters are used to generalise the construction of induced modules.

To simplify the exposition, we work with a restricted Lie algebra $(L,\p)$.  To apply the results to a general Lie algebra, as is the case for characters, we have to embed the algebra in a $p$-envelope, and the clusters obtained depend on that embedding.

\section{Preliminaries}
In the following, $(L,\p)$ is a restricted Lie algebra over the field $F$, $\bF$ is the algebraic closure of $F$ and $\bL = \bF \otimes_F L$ is the algebra obtained by extension of the field.  A character of $L$ is an $F$-linear map $c : L \to \bF$.  If $\{e_1, \dots, e_n\}$ is a basis of $L$, then $c$ can be expressed as a linear form $c(x) = \sum a_ix_i$ for $x = \sum x_ie_i$, where $a_i \in \bF$.  If $\alpha$ is an automorphism of $\bF/F$, that is, an automorphism of $\bF$ which fixes all elements of $F$, then $c^\alpha$ is the character $c^\alpha(x) = \sum a_i^\alpha x_i$ and is called a conjugate of $c$.   We do not distinguish in notation between $c:L \to \bF$ and its linear extension $\bL \to \bF$.  We denote by $F[c]$ the field $F[a_1, \dots, a_n]$ generated by the coefficients $a_i$.  It is the field generated by the $c(x)$ for all $x \in L$ and is independent of the choice of basis. 

If $V$ is an $L$-module, then $\bV$ is the $\bL$-module $\bF \otimes_F V$.  The action of $x \in L$ on $V$ is denoted by $\rho(x)$.  The module $V$ has character $c$ if $(\rho(x)^p-\rho(x^\scp))v = c(x)^pv$ for all $x \in L$ and all $v \in V$.

In the universal enveloping algebra $U(L)$, the element $x^p -x^\scp$ is central.  (See Strade and Farnsteiner \cite{SF}, p. 203.)  For the module $V$ giving the representation $\rho$, we put $\phi_x = \rho(x)^p - \rho(x^\scp)$.  We then have $[\phi_x, \rho(y)] = 0$ for all $x,y \in L$.

\begin{lemma} \label{psemi} The map $\phi : L \to \End(V)$ defined by $\phi_x(v) = (\rho(x^p) - \rho(x^\scp))v$ is $p$-semilinear.
\end{lemma}

\begin{proof}  In the universal enveloping algebra $U(L)$, we have 
$$(a+b)^p = a^p + b^p + \sum_{i=1}^{p-1} s_i(a,b)$$
(see Strade and Farnsteiner \cite{SF} p. 62 equation (3),) and 
$$(a+b)^\scp = a^\scp + b^\scp + \sum_{i=1}^{p-1}s_i(a,b).$$
(See Strade and Farnsteiner, p. 64 Property (3).)  Putting these together, we have 
$$(a+b)^p - (a+b)^\scp = a^p +b^p - a^\scp -b^\scp.$$
It follows that $\phi_{a+b} = \phi_a + \phi_b$.  Clearly, $\phi_{\lambda a} = \lambda^p \phi_a$.
\end{proof}

\begin{remark} \label{Lonly} In the decomposition of $\bV$ given by Theorem \ref{SFdecomp}, the summand corresponding to the character $c$ is 
$$\{v \in \bV \mid (\phi_x - c(x)^p1)^r v = 0 \text{ for some } r \text{ and all } x \in \bL\}.$$  By Lemma \ref{psemi}, we need only consider those $x \in L$, or indeed, in some chosen basis of $L$.
\end{remark}

\section{Clusters}

\begin{definition}  The  cluster $\cl(V)$ of an $L$-module $V$ is the set of characters of the composition factors of  the $\bL$-module $\bV =\bF \otimes_F V$.
\end{definition}

\begin{lemma} \label{conj} Suppose $c \in \cl(V)$.  Then the conjugates $c^\alpha$ of $c$ are in $\cl(V)$.
\end{lemma} 

\begin{proof}  Let $A/B$ be a composition factor of $\bV$ and let $\{v_1, \dots, v_k\}$ be a basis of $V$. The action $\rho(x)$ of $x \in L$ on $V$ and so also on $\bV$ is given in respect to this basis by a matrix $X$ with coefficients in $F$. An automorphism $\alpha$ maps $v=\lambda_1v_1 + \dots \lambda_kv_k$ to $v^\alpha = \lambda_1^\alpha v_1 + \dots + \lambda_k^\alpha v_k$.  Since $X^\alpha = X$, we have that $(xv)^\alpha = xv^\alpha$.  Thus $A^\alpha, B^\alpha$ are submodules of $\bV$ and $A^\alpha/B^\alpha$ is a composition factor.  The linear map $\phi_x = \rho(x)^p - \rho(x^\scp)$ also commutes with $\alpha$.  Thus from $\phi_x(a +B) = c(x)^pa +B$, it follows that $\phi_x(a^\alpha) + B^\alpha = c^\alpha(x)^pa^\alpha + B^\alpha$.  Thus $c^\alpha \in \cl(V)$.
\end{proof}
The statement $(xv)^\alpha=xv^\alpha$ may  suggest that $A/B$ and $A^\alpha/B^\alpha$ are isomorphic.  They are not.  The map $v \mapsto v^\alpha$ is not linear, as $(\lambda v)^\alpha = \lambda^\alpha v^\alpha$.

By Lemma \ref{conj}, a cluster $\cl(V)$ is a union of conjugacy classes of characters.

\begin{definition}  A cluster $\cl(V)$ is called simple if it consists of a single conjugacy class of characters.
\end{definition}

\begin{theorem} \label{simple}  Let $V$ be an irreducible $L$-module.  Then $\cl(V)$ is simple.
\end{theorem}

\begin{proof}  $\bV = \bF \otimes_FV$ has a direct decomposition $\bV = \sum_c \bV_c$ where the component $\bV_c$ is, by Remark \ref{Lonly}, the space 
$$\{v \in \bV \mid (\phi_x - c(x)^p1)^r v = 0 \text{ for all $x \in L$ and some } r\}.$$
Here, we may take for $r$ the length of a composition series of $\bV_c$, which is independent of $x$.   Let $c \in \cl(V)$.  Let $\bV_0 = \sum_\alpha \bV_{c^\alpha}$ where the sum is over the distinct conjugates $c^\alpha$.  Let $f_x(t) = \Pi_\alpha(t-c^\alpha(x)^p)$.  The coefficients of $f_x(t)$ are invariant under the automorphisms of $\bF/F$. Therefore for some $k$, we have that $f_x(t)^{p^k}$ is a polynomial over $F$.  As the field is not assumed to be perfect, this may  require $k > 0$.  Let $m_x(t)$ be the least power of $f_x(t)$ which is a polynomial over $F$.  Then, with $r$ the length of a composition series of $\bV_c$, we have
$$\bV_0 = \{v \in \bV \mid m_x(\phi_x)^r v =0 \text{ for all }x \in L\}.$$

The condition $m_x(\phi_x)^r v =0$ for all $x \in L$ may be regarded as a set of linear equations over $F$ in the coordinates of $v$.  These equations have a non-zero solution over $\bF$ since $\bV_c \ne 0$.  Therefore, they have a non-zero solution over $F$, that is, 
$$V_0 = \{v \in V \mid m_x(\phi_x)^r v =0 \text{ for all }x \in L\} \ne 0.$$
Since the $\rho_y$ commute with the $\phi_x$,  $V_0$ is a submodule of $V$.  Therefore $V_0=V$ and it follows that the set of conjugates of $c$ is the whole of $\cl(V)$.
\end{proof}

\section{The cluster decomposition}
We have seen that if $c \in \cl(V)$, then every conjugate $c^\alpha$ of $c$ is in $\cl(V)$.  It is convenient to expand our terminology and call any finite set $C$ of linear maps $c:L \to \bF$ a cluster if, for each $c \in C$, all conjugates of $c$ are in $C$.  With this expansion of our terminology, every cluster is a union of simple clusters.

\begin{theorem} \label{decomp}  Let $(L,\p)$ be a restricted Lie algebra and let $V$ be an $L$-module.  Suppose that $\cl(V)$ is the union $C_1 \cup \dots \cup C_k$ of the distinct simple clusters $C_i$.  Then $V = V_1 \oplus \dots \oplus V_k$ of submodules $V_i$ such that $\cl(V_i) = C_i$.
\end{theorem} 

\begin{proof}  By Theorem \ref{SFdecomp}, $\bV$ is the direct sum over the set of characters $c$ of submodules $\bV_c$ whose composition factors all have character $c$.  By Remark \ref{Lonly},  $\bV_c$ is the space annihilated by some sufficiently high power of $(\phi_x - c(x)^p1)$ for all $x \in L$.

Suppose $c \in C_i$.  Put $\bV_i = \sum_\alpha \bV_{c^\alpha}$.  Some power $m_x(t)$ of $\Pi_\alpha(t - c^\alpha(x)^p)$ is a polynomial over $F$, and $\bV_i$ is the space annihilated by $m_x(\phi_x)^r$ for all $x \in L$ and some sufficiently large $r$.  Put 
$$V_i = \{v \in V \mid (m_x(\phi_x)^rv=0 \text{ for all }x \in L\}.$$

The set of conditions $m_x(\phi_x)^rv=0$ for all $x \in L$ may be regarded as a set of linear equations over $F$ in the coordinates of $v$, so the $F$-dimension of its solution space $V_i$ in $V$ is equal to the $\bF$-dimension of its solution space $\bV_i$ in   $\bV$.  It follows that $V = \oplus_i V_i$.
Clearly $V_i$ is a submodule of $V$ and $\cl(V_i) = C$. 
\end{proof}

\begin{theorem} \label{sndec} Suppose that $S \sn L$ and let $V$ be an $L$-module.  Then the components of the cluster decomposition $V = \oplus_C V_C$ with respect to $S$ are $L$-submodules. 
\end{theorem}

\begin{proof}  Although $S$ need not be a restricted algebra, it is embedded in the restricted algebra $(L,\p)$ and the components are defined using the operation $\p$.  There exists a series $S = S_0 \id S_1 \id \dots \id S_n = L$.  We use induction over $i$ to prove that $V_C$ is an $S_i$-module.  Take $x \in S_i$ and consider $(xV_C+V_C)/V_C$.  For $s \in S$ and $v \in V_C$, we have $s(xv) = x(sv) + [s,x]v$.  But $[s,x] \in S_{i-1}$, so $[s,x]v \in V_C$.  Thus the map $v \mapsto xv+V_C \in (xV_C+V_C)/V_C$ is an $S$-module homomorphism.  Thus the character of every composition factor of $\bF \otimes_F((xV_C+V_C)/V_C)$ is in $C$, which implies that $xV_C \subseteq V_C$.
\end{proof}

\begin{remark}  The decomposition given by Theorem \ref{sndec} depends on the $p$-operation, not merely on the algebra $S$.  Changing the $p$-operation may change the decomposition as is shown by the following example.  This opens the possibility that, where the minimal $p$-envelope of $S$ has non-trivial centre, judicious variation of the $p$-operation may give useful different direct decompositions.
\end{remark}

\begin{example}  Let $L = \langle a_1, a_2 \mid [a_1,a_2] = 0 \rangle$ and let $V = \langle v_1, v_2 \rangle$ with $a_iv_i = v_i$ and $a_iv_j=0$ for $i \ne j$.  With $a_1^\scp = 0$ and $a_2^\scp = -a_1$, $V$ has the character $c$ with $c(a_1) = 0$ and $c(a_2) = 1$.  The cluster decomposition with respect to $(L,\p)$ is simply $V = V_c$.  However, with the $p$-operation $\pd$ with $a_i^{\scp'}= 0$, the submodule $\langle v_1 \rangle$ has character $c_1$ with $c_1(a_1) = 1$ and $c_1(a_2) = 0$, while $\langle v_2\rangle$ has character $c_2$ with $c_2(a_1)=0$ and $c_2(a_2) = 1$.  This gives the cluster decomposition $V = V_{c_1} \oplus V_{c_2}$.
\end{example}

\section{$\fF$-central and $\fF$-eccentric modules}
Let $\fF$ be a saturated formation of soluble Lie algebras over $F$.  Comparing Theorem \ref{sndec} with \cite[Lemma 1.1]{extras} suggests a further relationship between clusters and saturated formations beyond that of \cite[Theorem 6.4]{extras}.

\begin{theorem} \label{Fcent}  Let $\fF$ be a saturated formation and suppose $S \in \fF$.  Let $(L,\p)$ be a $p$-envelope of $S$ and suppose that $z^\scp = 0$ for all $z $ in the centre of $L$.  Let $V,W$ be irreducible $S$-modules.  Suppose that $\cl(V) = \cl(W)$.  Then $V,W$ are either both $\fF$-central or both $\fF$-eccentric.
\end{theorem}

\begin{proof}  Suppose to the contrary, that $V$ is $\fF$-central and that $W$ is $\fF$-eccentric. By \cite[Theorem 2.3]{B6}, $\Hom(V,W)$ is $\fF$-hypereccentric.  But from \cite[Theorem 5.2.7]{SF} it follows that the characters of the composition factors of $\Hom(\bV,\bW)$ are all of the functions $c_2-c_1$ where $c_1 \in \cl(V)$ and $c_2 \in \cl(W)$.  Since $\cl(V) = \cl(W)$, we have that $0 \in \cl(\Hom(V,W))$.  

By assumption, we have that $z^\scp = 0$ for all $z$ in the centre of $L$.  As $(L,\p)$ is a $p$-envelope of $S$, we have $S \ideq L$. By \cite[Theorem 6.4]{extras}, a composition factor $X$ of $\Hom(V,W)$ with $\cl(X) = \{0\}$ is $\fF$-central, contrary to $\Hom(V,W)$ being $\fF$-hypereccentric.
\end{proof}

\section{$C$-induced modules}
Let $(L,\p)$ be a restricted Lie algebra over the perfect field $F$ and let $S$ be a $\p$-subalgebra of $L$.  Let $W$ be an $S$-module and let $C$ be a cluster of characters of $L$ whose restriction to $S$ is $\cl(W)$.  We require that distinct members of $C$ have distinct restrictions to $S$ in which case, we say that $C$ restricts simply to $S$.  

Note that, given a simple cluster $C_S$ of $S$, we can easily construct a cluster $C$ of $L$ which restricts simply to $C_S$.  We take a cobasis $\{e_1, \dots, e_n\}$ of $S$ in $L$, that is, a basis of some subspace complementary to $S$.  A character $c: S \to \bF$ can be extended to $L$ by assigning arbitrarily the values $c(e_i) \in \bF$.  If these are chosen in $F[c]$, then any automorphism which fixes the given $c$ also fixes its extension. 

We want to apply the construction of $c$-induced modules (see Strade and Farnsteiner \cite[Section 5.6]{SF}) to the $c$-components $\bW_c$ of $\bW$.  This  construction only works for modules with character $c$.  Every compostion factor of $\bW_c$ has character $c$, but $\bW_c$ itself need not.  This leads to the following definition.

\begin{definition}  We say that the $S$-module $W$ is amenable (for induction) if, for all $c \in \cl(W)$, $\bW_c$ has character $c$.
\end{definition}

Note that if $\bW_c$ has character $c$, then for each conjugate $c^\alpha$ of $c$, $\bW_{c^\alpha}$ has character $c^\alpha$ .  It would be nice to have a way of determining if a module $W$ is amenable which does not require analysis of $\bW$.  The following lemmas achieve that.  

\begin{lemma} \label{ambasis} Let $\{s_1, \dots, s_n\}$  be a basis of $S$ and let $W$ be an $S$-module.  Let $m_i(t)$ be the minimum polynomial of $\phi_{s_i}$. Then $W$ is amenable if and only if for all $i$, $\gcd(m_i(t),m'_i(t)) = 1$.
\end{lemma}

\begin{proof}   The module  $W$ is amenable if and only if, for all $c \in \cl(W)$ and all $i$, we have $(\phi_{s_i} - c(s_i)^p1)\bW_c = 0$.  So $W$ is amenable if and only if for all $i$, in $\bF[t]$, $m_i(t)$ has no repeated factors, that is, if and only if $\gcd(m_i(t),m'_i(t)) = 1$.  As the calculation of $\gcd(m_i(t),m'_i(t))$ in $F[t]$ is the same as in $\bF[t]$, the result follows.
\end{proof}

\begin{lemma} \label{irredamen}  Let $W$ be an irreducible $S$-module.  Then $W$ is amenable.
\end{lemma}

\begin{proof}  For any $s \in S$, $s^p - s^\scp$ is in the centre of the universal enveloping algebra of $S$ and so, for any representation $\rho$ of $S$, we have $[\rho(s_1)^p - \rho(s_1^\scp), \rho(s_2)] = 0$ for all $s_1,s_2 \in S$.  For $c \in \cl(W)$, put $f_s(t) = \Pi (t-c^\alpha(s)^p)$ where the product is taken over the distinct conjugates of $c(s)$.  Then $f_s(t)$ is a polynomial over $F$, and $\rho(s_2)$ commutes with $f_{s_1}(\phi_{s_1})$ for all $s_1,s_2 \in S$.  Thus $W_0 = \{w \in W \mid f_s(\phi_s) w = 0\text{ for all } s \in S\}$ is a submodule of $W$.  The conditions $f_s(\phi_s)w = 0$ are linear equations over $F$ with non-zero solutions over $\bF$ and so have non-zero solutions over $F$.  Thus $W_0 \ne 0$ which implies $W_0 = W$.
\end{proof}

As the construction being developed can be applied separately to each direct summand of $W$, we suppose that $C$ is simple.  Take a basis $\{b^1, \dots, b^k\}$ of $W$.  Corresponding to each $c \in C$, we have a component $\bW_c$ of $\bW = \oplus_\alpha \bW_{c^\alpha}$.  For each $w \in \bW$, we have $w = \sum_c  w_c$ with $w_c \in \bW_c$.

\begin{lemma} \label{inv}  Let $w = \sum \lambda_i b^i \in \bW$.  Then $w$ is invariant under the automorphisms of $\bF/F$ if and only if the $\lambda_i \in F$, in which case, $(w_c)^\alpha = w_{c^\alpha}$.  Further, $sw$ is also invariant for all $s \in S$.
\end{lemma}

\begin{proof}  If $w = \sum \lambda_i b^i $ is invariant, then $\lambda_i$ is invariant.  Since $F$ is perfect, this implies $\lambda_i \in F$.  If $\lambda_i \in F$ for all $i$, then clearly $w$ is invariant.  As $W$ is an $S$-module, also $sw$ is invariant.  The action of $\alpha$ permutes the $\bW_c$ and does not change the direct decomposition.  That $(w_c)^\alpha = w_{c^\alpha}$ follows.
\end{proof}

Suppose that $C$ is a simple cluster of characters of $L$ which restricts simply to $S$ and that $W$ is an amenable $S$-module with $\cl(W) = C|S$. For each $c \in C$, we form the $c$-reduced enveloping algebras $u(L,c)$ and $u(S,c)$.  (See Strade and Farnsteiner \cite{SF} page 226.)  Since $W$ is amenable, we can construct the $c$-induced $\bL$-modules
$$\bV_c= \ind^{\bL}_{\bS}(\bW_c, c) = u(\bL,c) \otimes_{u(\bS,c)} \bW_c$$
and put $\bV = \oplus_c \bV_c$.   From  $\bV$, we shall select an $F$-subspace $V$ with $\bF \otimes_F V=\bV$, which we shall show to be an $L$-module with $\cl(V) = C$. 

For $x,y \in L$ in the following,  we need to distinguish  their product in the associative algebra $u(L,c)$ from their product in the Lie algebra.  We denote the Lie algebra product by $[x,y]$.  Take a cobasis $\{e_1, \dots, e_n\}$ for $S$ in $L$.  Then elements $e_1^{r_1}e_2^{r_2} \dots e_n^{r_n} \otimes w_c$ with $r_i \le p-1$ and $w_c \in \bW_c$ span $\bV_c$.  To simplify the notation, we write $e(r)$ for $e_1^{r_1}e_2^{r_2} \dots e_n^{r_n}$.    For an element $w = \sum_c w_c \in \bW$, it is convenient to abuse notation and write  $e(r) \otimes w$ for the element $\sum_c e(r)\otimes w_c$.  It should be remembered that in this sum, the $e(r)$ come from different algebras $u(\bL,c)$ with different multiplication, and that the tensor products are over different algebras $u(\bS,c)$.

Any element $w_c \in \bW_c$ is an $\bF$-linear combination of the $b^i$, so an element of $\bV_c$ is expressible as an $\bF$-linear combination of the $e(r) \otimes b^i$.  It follows that the   $e(r) \otimes b^i$ form a basis of $\bV$.  An automorphism $\alpha$ maps $e(r) \otimes w$ to $e(r) \otimes w^\alpha$.  Thus the invariant elements of $\bV$ are the $F$-linear combinations of the basis.

\begin{lemma} \label{xinv}  Let $v \in \bV$ be invariant.  Then $xv$ is invariant for all $x \in L$.
\end{lemma}

\begin{proof}  We use induction over $k$ to show that $x_1 \dots x_k \otimes b^i$ is invariant for all $x_1 , \dots, x_k \in L$.  The result then follows trivially.

For $s \in S$, we have $s(1 \otimes b^i) = 1 \otimes sb^i$ which is invariant by Lemma \ref{inv}.  For $e_j$, we have $e_j(1 \otimes b^i) = e_j \otimes b^i$ which is invariant.  Note that in this case, the multiplication is the same in all the $u(\bL,c)$.  Thus the result holds for $k=1$.

Suppose that $k > 1$.  We express each of the $x_t$ as a linear combination of the $e_j$ and an element of $S$.  We then use the commutation rules $xy -yx = [x,y]$ to move each factor to its correct position, giving a sum of terms of the form $e_1^{r_1}e_2^{r_2} \dots e_n^{r_n} s_1 \dots s_m\otimes b^i$, but with the $r_j$ not restricted to be less than $p$.  The terms coming from a commutator $[x,y]$ all have fewer than $k$ factors and so are invariant.  Any elements of $S$ at the end moves past the tensor product, giving  $e_1^{r_1}e_2^{r_2} \dots e_n^{r_n} \otimes s_1 \dots s_m b^i$.  By Lemma \ref{inv}, $1 \otimes s_1 \dots s_m b^i$ is invariant and since, in this case, $e_1^{r_1} \dots e_n^{r_n}$ has fewer than $k$ factors, the term is invariant.   Thus we are left to consider terms of the form $e_1^{r_1} \dots e_n^{r_n} \otimes b^i$.  If $r_j <p$ for all $j$, then the term is one of our basis elements and so is invariant.

Suppose that for some $j$,  we have $r_j \ge p$.  Then we must separate the summands. The term can be written in the form $e e_j^p e' \otimes b^i$ where $e, e'$ are strings of cobasis elements.  In the algebra $u(\bL, c^\alpha)$, we have $e_j^p = e_j^\scp + c^\alpha(e_j)^p1$.  Thus
$$e e_j^p e' \otimes b^i = e e_j^\scp e' \otimes b^i + \sum_\alpha ee' \otimes c^\alpha(e_j)^p b^i_{c^\alpha}.$$ 
But $e e_j^\scp e' \otimes b^i$  has fewer than $k$ factors and so invariant. As $ \sum_\alpha 1 \otimes c^\alpha(e_j)^p b^i_{c^\alpha}$ is invariant and $ee'$ has fewer than $k$ factors, it follows that $e e_j^p e'\otimes b^i$ also is invariant.
\end{proof}

We define the $C$-induced module $\ind^L_S(W,C)$ to be the $F$-subspace of invariant elements of $\bV$.  By Lemma \ref{xinv}, it is an $L$-module.  Clearly $\cl(\ind^L_S(W,C)) = C$.

To illustrate this, we calculate a simple example.
\begin{example}  Let $F$ be the field of $3$ elements and let $L = \langle x,y \mid [x,y] = y \rangle$.  Putting $x^\scp = x$ and $y^\scp = 0$ makes this a restricted Lie algebra.  We take $S = \langle x \rangle$ and $W = \langle b^1,b^2\rangle$ with $xb^1 = b^2$ and $xb^2 = -b^1$.  Over $\bF$, we have $\bW = \bW_1 \oplus \bW_2$ with $\bW_1 = \langle -b^1-ib^2 \rangle$ and $\bW_2 = \langle -b^1 + ib^2 \rangle$, where $i \in \bF, i^2 = -1$.  Denote the action of $S$ on $W$ by $\rho$.  Then $\rho(x)(-b^1-ib^2) =-i(-b^1-ib^2)$ and $\rho(x)(-b^1+ib^2) = i(-b^1+ib^2)$.

We have $(\rho(x)^p - \rho(x^\scp))(-b^1-ib^2) = ((-i)^3 - (-i))(-b^1-ib^2) = -i(-b^1-ib^2)$.  Thus the character $c_1$ of $\bW_1$ must have $c_1(x)^3 = -i$, so $c_1(x) = i$.    Similarly, we have $c_2(x) = -i$.  As distinct conjugates of a character on $L$ in $C$ must have distinct restrictions to $S$, $c_1(y) \in F[i]$.  Put $c_1(y) = \lambda = \alpha + i\beta$ where $\alpha, \beta \in F$.  Then $c_2(y)=\bar{\lambda}$.  Note that in $u(\bL,c_1)$, $y^3 = \lambda^3 = \bar{\lambda}$.  In both the algebras, $xy = y + yx$ and 
$xy^2 =  (y+yx)y = y^2 + y(xy) = -y^2 + y^2x$.

In the notation used above, we have $b^1_{c_1} = -b^1-ib^2$ and $b^1_{c_2} = -b^1+ib^2$, while for $b^2$, we have $b^2_{c_1} = ib^1-b^2$ and $b^2_{c_2} = -ib^1-b^2$.  The induced module $V = \ind^L_S(W,C)$ has basis the six elements $v^r_j = y^r \otimes b^j$ for $r = 0 , 1, 2$ and $j = 1,2$.  We calculate the actions of $x,y$ on these elements.
\begin{align*}
xv_1^0 &= x(1 \otimes b^1) = v_2^0 & xv_2^0&=x(1\otimes b^2)=-v_1^0\\  
xv_1^1 &= (y+yx)\otimes b^1 =v_1^1+v_2^1 & xv_2^1&=(y+yx)\otimes b^2=v_2^1 - v_1^1\\ 
xv_1^2&=(-y^2+y^2x)\otimes b^1= -v_1^2 + v_2^2 & xv_2^2&= (-y^2+y^2x)\otimes b^2=-v_2^2-v_1^2\\
yv_1^0&=v_1^1 & yv_2^0&=v_2^1\\
yv_1^1&=v_1^2 & yv_2^1&=v_2^2
\end{align*}
The calculations of $yv_j^2$ are more complicated.
\begin{equation*}
\begin{split}
yv_1^2&= y^3\otimes (b_{c_1}^1 + b_{c_2}^1)= 1\otimes(\bar{\lambda} b_{c_1}^1 + \lambda b_{c_2}^1)\\
&=1 \otimes\bigl((\alpha -i\beta)(-b^1-ib^2) + (\alpha+i\beta)(-b^1+ib^2)\bigr)\\
&=1 \otimes (\alpha b^1 + \beta b^2) = \alpha v_1^0 + \beta v_2^0\\
yv_2^2&= y^3 \otimes(b_{c_1}^2 + b_{c_2}^2)= 1\otimes(\bar{\lambda} b_{c_2}^1 + \lambda b_{c_2}^2)\\
&=1 \otimes \bigl((\alpha - i\beta)(ib^1-b^2) + (\alpha+i\beta)(-ib^1-b^2)\bigr)\\
&=1 \otimes(-\beta b^1+\alpha b^2) =  -\beta v_1^0 + \alpha v_2^0
\end{split}
\end{equation*} 
\end{example}

\begin{remark}  As noted earlier, in the notation used above, if we are given an amenable $S$-module $W$ with simple cluster $C_S$, we can construct a simple cluster $C$ on $L$ which restricts simply to $C_S$ by choosing arbitrarily the $c_1(e_i)$ in $F[c_1]$.  If we choose the $c_1(e_i)$ in $F$, then we have $c_j(e_i) = c_1(e_i)$ for all $j$.  This simplifies the calculation of the action on the induced module as we then have $e_i^p = e_i^\scp + c(e_i)^p1$ in all the algebras $u(\bL,c_i)$ and it follows that $e_i^pb^j = (e_i^\scp +c(e_i)^p1)b^j$ can be calculated without using the character decomposition of $\bW$.  That the action of $x \in L$ on a basis element $e(r)\otimes b^i$ can be calculated without using the decomposition follows by an induction  as in the proof of Lemma \ref{xinv}. It thus becomes possible to calculate the action on $\ind_S^L(W,C)$ without having to determine the eigenvalues of the $\phi_{s_i}$.  In the above example, if we take $\beta = 0$, then the calculations of $yv_1^2$ and $yv_2^2$ simplify to $yv_i^2 = y^3b^i = \alpha^3b^i = \alpha v_i^0$.
\end{remark}

\begin{remark}  Denote the category of amenable $L$-modules with cluster $C$ by $\ammod(L,C)$. The restriction functor $\res_L^S: \ammod(L,C) \to \ammod(S,C|S)$ sends an $L$-module $V$ to $V$ regarded as an $S$-module. Suppose that $C$ restricts simply to $S$.  Then $\ind^L_S(\ , C)$ is a functor $\ammod(S,C|S) \to \ammod(L,C)$.   In the special case where $C=\{c\}$, $\ind^L_S(\ ,c)$ is a left adjoint to $\res^S_L$ by Strade and Farnsteiner \cite[Theorem 5.6.3] {SF}.  Applying this to the $\bW_c$ in the general case gives that $\ind^L_S(\ ,C)$ is a left adjoint to $\res^S_L$.
\end{remark}

\bibliographystyle{amsplain}

\end{document}